\documentclass[12pt,reqno]{amsart}
\usepackage{amssymb}
\usepackage{amsmath}
\usepackage{amsthm}
\usepackage[left=3cm,top=3cm,right=3cm,bottom=3cm]{geometry}
\usepackage{hyperref}
\usepackage[usenames]{color}
\usepackage{enumerate}
\usepackage[normalem]{ulem} 
\usepackage{graphicx}
\usepackage{eucal}

\usepackage{tikz}
\usetikzlibrary{positioning,chains,fit,shapes,calc}

\newtheoremstyle{plain}%
    {8pt plus2pt minus4pt}%
    {8pt plus2pt minus4pt}%
    {\itshape}%
    {}%
    {\bfseries\scshape}%
    {}%
    {6pt}
    {}%

\usepackage{etoolbox}
\makeatletter
\patchcmd{\@maketitle}
  {\ifx\@empty\@dedicatory}
  {\ifx\@empty\@date \else {\vskip3ex \centering\footnotesize\@date\par\vskip1ex}\fi
   \ifx\@empty\@dedicatory}
  {}{}
\patchcmd{\@adminfootnotes}
  {\ifx\@empty\@date\else \@footnotetext{\@setdate}\fi}
  {}{}{}
\makeatother

\begin{document}
\newcommand{\weighted}[1]{\mbox{2-WEIGHTED}^{{(#1)}}}
\newcommand{\Whp}{{\textit{W.h.p. }}}
\newcommand{\whp}{{\textit{w.h.p. }}}
\newcommand{\bin}{\textrm{Bin}}
\newcommand{\po}{\textrm{Po}}
\newcommand{\codeg}{\textrm{codeg}}
\newcommand{\E}{\mathrm{E}}
\newcommand{\V}{\mathrm{Var}}
\newcommand{\G}{\mathbb{G}}
\newcommand{\RT}{\textup{\textbf{RT}}}
\newcommand{\f}{\textup{\textbf{f}}}
\newcommand{\z}{\textrm{\textbf{z}}}
\newcommand{\ex}{\textrm{\textbf{ex}}}
\newcommand{\sqbs}[1]{\left[ #1 \right]}
\newcommand{\of}[1]{\left( #1 \right)}
\newcommand{\bfrac}[2]{\of{\frac{#1}{#2}}}
\renewcommand{\l}{\ell}
\newcommand{\sm}{\setminus}
\newcommand{\mc}[1]{\mathcal{#1}}
\newcommand{\h}[1]{\mathbb{H}^{({#1})}}
\newcommand{\hh}{\mathbb{H}}
\newcommand{\g}{\mathbb{G}}
\newcommand{\dist}{\mathrm{dist}}
\newcommand{\la}{\lambda}
\renewcommand{\k}{\kappa}
\newcommand{\eps}{\epsilon}
\renewcommand{\P}{\mathbb{P}}

\newcommand{\rbrac}[1]{\left( #1 \right)}
\newcommand{\rfrac}[2]{\rbrac{\frac{#1}{#2}}}
\newcommand{\sbrac}[1]{\left[ #1 \right]}
\newcommand{\cbrac}[1]{\left\{ #1 \right\}}

\newtheorem*{theorem*}{Theorem}
\newtheorem{theorem}{Theorem}[section]
\newtheorem{lemma}[theorem]{Lemma}
\newtheorem{definition}[theorem]{Definition}
\newtheorem{conjecture}[theorem]{Conjecture}
\newtheorem{proposition}[theorem]{Proposition}
\newtheorem{claim}[theorem]{Claim}
\newtheorem{fact}[theorem]{Fact}
\newtheorem{corollary}[theorem]{Corollary}
\newtheorem{observation}[theorem]{Observation}
\newtheorem{problem}[theorem]{Open Problem}

\title[]{On The Number of Alternating Paths in Random Graphs}
\author{Patrick Bennett, Ryan Cushman, Andrzej Dudek}
\address{Department of Mathematics, Western Michigan University, Kalamazoo, MI}
\email{\tt \{patrick.bennett,\;ryan.cushman,\;andrzej.dudek\}@wmich.edu}
\thanks{The first author was supported in part by Simons Foundation Grant \#426894 and by funds from the Faculty Research and Creative Activities Award, Western Michigan University.
The third author was supported in part by Simons Foundation Grant \#522400.}

\begin{abstract}

In the noisy channel model from coding theory, we wish to detect errors introduced during transmission by optimizing various parameters of the code. Bennett, Dudek, and LaForge framed a variation of this problem in the language of alternating paths in edge-colored complete bipartite graphs in 2016. Here, we extend this problem to the random graph $\G(n,p)$. We seek the \emph{alternating connectivity}, $\kappa_{r,\ell}(G)$, which is the maximum $t$ such that there is an $r$-edge-coloring of $G$ such that any pair of vertices is connected by $t$ internally disjoint and alternating (i.e. no consecutive edges of the same color) paths of length~$\ell$. 
We have three main results about how this parameter behaves in $\G(n,p)$ that basically cover all ranges of~$p$: one for paths of length two, one for the dense case, and one for the sparse case. For paths of length two, we found that $\kappa_{r,\ell}(G)$ is essentially the codegree of a pair of vertices. For the dense case when $p$ is constant, we were able to achieve the natural upper bounds of minimum degree (minus some intersection) or the total number of disjoint paths between a pair of vertices. For the sparse case, we were able to find colorings that achieved the natural obstructions of  minimum degree or (in a slightly less precise result) the total number of paths of a certain length in a graph. We broke up this sparse case into ranges of $p$ corresponding to when $\G(n,p)$ has diameter $k$ or $k+1$. We close with some remarks about a similar parameter and a generalization to pseudorandom graphs.

\end{abstract}

\date{\today}

\maketitle


\section{Introduction}

\noindent An alternating path is a path with adjacent edges having distinct colors. Originally studied by Bollob\'as and Erd\H{o}s~\cite{BE} in the context of finding alternating Hamiltonian cycles in complete graphs, alternating paths have been widely studied~(e.g.~\cite{AFR,AG, JG, CD, DFR, S}). Along with being interesting objects in their own right, alternating paths can also be used to encapsulate certain parameters of codes in coding theory.

One such fundamental parameter of a code $C\subseteq [r]^m$ is the minimum \emph{Hamming distance} between codewords $\mathbf{x}$ and $\mathbf{y}$ in $C$. This distance between codewords is defined as the number of positions where they differ: 
\[
\dist(\mathbf{x}, \mathbf{y}) = \big{|} \{ i: 1\le i\le m,\ \mathbf{x}(i)\neq \mathbf{y}(i)\} \big{|}.
\]

\noindent In the noisy channel model of coding theory, we wish to detect errors introduced during transmission by optimizing various parameters of the code. Codes with large minimum Hamming distance between codewords allow for greater detection and correction of errors. At the same time, one wants to be able to send many messages across this channel. Thus, one may wish to ask what the maximum $n$ is such that there exists a code $C\subseteq [r]^m$ of $n$ codewords and minimum Hamming distance $t$. We call this number $\alpha_r(m, t)$. (See~\cite{P} for more details concerning coding theory.)

We now encode this problem in terms of alternating paths. Let $K_{m,n}$ be a complete bipartite graph on vertex set $[m]\cup [n]$ and $c: E(K_{m,n}) \rightarrow [r]$ be an $r$-edge-coloring of $K_{m,n}$ with the property that every pair of vertices in~$[n]$ is connected by at least~$t$ alternating paths of length 2 (with 3 vertices). We can represent this coloring as a collection of $n$ vectors of length $m$ with entries in $[r]$ in the following way: for a vertex $v \in [n]$, define the vector $\mathbf{c}_v\in [r]^m$ by $\mathbf{c}_v(u) = c(\{v,u\})$ for $u \in [m]$. Then $C = \{\mathbf{c}_v : v\in [n] \}$ fully encodes the edge coloring of $K_{m,n}$ since every edge will belong to a unique vector which will determine its color. 

But if we take a path of length 2 of the form $v-u-w$ for $v,w \in [n]$, notice that this is alternating if and only if $c(\{v,u\}) \neq c(\{u,w\})$, that is, $\mathbf{c}_v(u) \neq \mathbf{c}_w(u)$. So the number of alternating paths between $v$ and $w$ is exactly the Hamming distance between $\mathbf{c}_v$ and $\mathbf{c}_w$. And since $c$ has the property that every pair of vertices in the same partite set is connected by at least $t$ alternating paths of length 2, it follows that $C$ has minimum Hamming distance $t$. Therefore we may conclude that $|C| = n \le \alpha_r(m,t)$.  

So determining $\alpha_r(m, t)$ is equivalent to finding the largest $n$ such that there is an $r$-coloring of the edge set of $K_{m, n}$ with the property that any pair of vertices in~$[n]$ has at least $t$ alternating paths of length $2$ connecting them. Also notice that all such paths are internally disjoint. 

Bennett, Dudek and LaForge~\cite{BDL} used this characterization to study a slightly different problem. Instead of fixing the alphabet, word length, and $t$ and asking for the largest possible code, they fixed the alphabet, word length, and code size and asked for the largest possible $t$. In addition, they generalized the length of the alternating paths from $2$ to any constant length $2k$. Their object of study was $\kappa_{r,2k}(m,n)$, the maximum $t$ such that there is an $r$-coloring of the edges of $K_{m,n}$ such that any pair of vertices in class of size~$n$ is connected by $t$ internally disjoint and alternating paths of length~$2k$.

They showed that for fixed $r\ge2$ and $n\ge m\gg \log n$
\[
\kappa_{r,2}(m,n) \sim \left( 1-\frac{1}{r}\right) m
\]
(see the last paragraph of this section regarding the notation $\gg$ and $\sim$).
They also proved that for any fixed $r\ge 2$, $k\ge 2$ and $n\ge m\gg 1$
\[
\kappa_{r,2k}(m,n) \sim  \frac{m}{k}.
\]

Besides these results, the authors also proposed the concept of alternating connectivity, of which this paper is mainly concerned. 
This parameter was in part inspired by the work of Espig, Frieze, and Krivelevich \cite{EFK}, who found conditions under which a random graph with randomly two-colored edges has an alternating path joining any pair of vertices.
\emph{Alternating connectivity}
$\kappa_{r,\ell}(G)$ was defined
to be the maximum $t$ such that there is an $r$-edge-coloring of $G$ such that any pair of vertices is connected by $t$ internally disjoint and alternating paths of length~$\ell$.

For complete graphs, it was shown that 
\[
\kappa_{r,2}(K_n)\sim (1-1/r) n
\quad\text{ and }\quad
\kappa_{r,\ell}(K_n)\sim n/(\ell-1)
\]
for any $r\ge 2$ and $\ell\ge 3$. Motivated by this progress, in this paper we study 
the alternating connectivity of the random graph $\G(n,p)$.

We have three main results concerning this parameter for $\G(n, p)$. The first result is concerned with the case when we seek paths of length two. Here, our answer is essentially the same as the expected codegree between two vertices in $\G(n, p)$. 
\begin{theorem} \label{thm:kappa_r2}
Let $p \gg \sqrt{\log n/n}$ and $G = \G(n,p)$ and let $r$ be an integer. Then, \whp
\[
\kappa_{r,2}(G) \sim \left( 1-\frac{1}{r}\right) np^2.
\]
\end{theorem}
\noindent 
Observe that if $p \ll \sqrt{\log n/n}$ then the diameter of $G$ is at least three. Therefore, in this sense the above theorem is optimal.

After this, we consider $\ell \ge 3$ by looking at a ``dense'' case and a ``sparse'' case. For the ``dense'' case, when $p$ is a constant, we have found that we have two main obstructions. Clearly, we cannot have more paths between two vertices than the minimum degree. However, the neighborhoods of two vertices \whp~have an intersection of size $\sim np^2$, so the most we could hope for here is $np(1-p/2)$. On the other hand, the total number of paths of length $\ell$ between two vertices is bounded above by $n/(\ell -1)$. 
\begin{theorem}\label{thm:kappa-dense}
Let $ 0 < p < 1$ be a constant and $G = \G(n,p)$. Then,  for any integer $\ell \ge 3$ \whp 
\[
\kappa_{r,\ell}(G) \sim \min \left\lbrace\frac{n}{\ell -1},  np\left(1-\frac{p}{2}\right)\right\rbrace.
\]
\end{theorem}
Finally, we consider the ``sparse'' case when $p = o(1)$. Here, we have two main obstructions: in parts \eqref{thm:sparse:ii} and  \eqref{thm:sparse:i}, we are limited by the minimum degree; in part \eqref{thm:sparse:iii} we are limited by the total number of paths of length $k$. Here, we analyze $\kappa_{r,\ell}(G)$ in based on how close we are to the threshold for the diameter in $\G(n,p)$. 

\begin{theorem}\label{thm:kappa-sparse}
Suppose $G = \G(n,p)$ with $p=o(1)$ and $r\ge 2$ is an integer.
\begin{enumerate}[(i)]
\item\label{thm:sparse:ii} Let $k\ge 2$ be a positive integer such that $n^{1/k} \le np \le n^{1/(k-1)}$. If $\ell \ge k+2$, then 
\whp~we have $\kappa_{r,\ell}(G)\sim np$.
\item\label{thm:sparse:i} Let $k\ge 2$ be a positive integer such that $(n\log n)^{1/k} \ll np \le n^{1/(k-1)}$. If $\ell = k+1$, then 
\whp~we have $\kappa_{r,\ell}(G)\sim np$.
\item\label{thm:sparse:iii} Let $k\ge 3$ be a positive integer such that $(n\log n)^{1/k} \ll np \ll n^{1/(k-1)}$. If $\ell = k$, then 
\whp~we have $\kappa_{r,\ell}(G)= \Theta(n^{k-1}p^k)$.
\end{enumerate}
\end{theorem}

It is important to recall here the following result (Corollary 10.12 in~\cite{B}) that asserts that \whp~$\G(n, p)$ has diameter $k$ if $(np)^{k}/ n - 2 \log n \rightarrow \infty$ and $(np)^{k-1}/n - 2\log n \rightarrow -\infty$. Thus, when $p$ is in the range of part~\eqref{thm:sparse:i} and \eqref{thm:sparse:iii} \whp~$\G(n, p)$ has diameter $k$. But for the range in \eqref{thm:sparse:ii}, we have that \whp~$\G(n, p)$ has diameter $k$ or $k+1$. Note that Theorem \ref{thm:kappa-sparse} becomes weaker as we look for paths whose length is closer to the diameter, and in particular (in part \eqref{thm:sparse:iii}) when the path length equals the diameter we are only able to estimate $\k$ up to a constant factor.

We will need some straightforward lemmas about $\G(n,p)$ which will be proved in Section \ref{sec:lemmas}. Finally, we close with some remarks about the current state of $\lambda_{r,\ell}(G)$ (where we no longer require that paths are internally disjoint) and extending our results to pseudorandom graphs.

Throughout this paper all asymptotics are taken in~$n$ unless noted otherwise. Beside of the standard \emph{Big-$O$} and \emph{Little-$o$} notation we will also write  $\sim\!\!f(n)$ instead of $(1+o(1))f(n)$. For simplicity, we do not round numbers that are supposed to be integers either up or down; this is justified since these rounding errors are negligible to the asymptomatic calculations we will make. All logarithms are natural unless written explicitly.

\section{Auxiliary results for $\G(n,p)$} \label{sec:lemmas}

Throughout the paper we will be using the following forms of Chernoff's bound (see, e.g., \cite{JLR}). Let $X\sim \bin(n,p)$ and $\mu = \E(X)$. Then, for all $0<\delta<1$
\begin{equation}\label{Chernoff_upper}
\Pr(X \ge (1+\delta) \mu) \le \exp(-\mu \delta^2/3)
\end{equation}
and
\begin{equation}\label{Chernoff_lower}
\Pr(X \le (1-\delta)\mu) \le \exp(-\mu \delta^2/2). 
\end{equation}

Several times we will also use the well-known inequalities:
\begin{equation}\label{eq:ineqs}
1-x \le e^{-x} \text{ for any $x$ and } 1-x/2 \ge e^{-x} \text{ for } 0\le x\le 1.
\end{equation}

We will rely on several straightforward lemmas about matchings in random graphs. The first one is about almost perfect matchings in the dense random graph. A similar result for bipartite random graphs (with almost identical proof) was obtained in~\cite{BDL}. 

\begin{lemma}\label{lem:matchings}
Let $0 < \alpha, p < 1$ be constants and $\G(n,p)$ be a random graph on set of vertices $V$. Then, \whp~for any subsets $A,B \subseteq V$ with $|A|=|B|=\alpha n$, there exists a matching between them of size $\alpha n(1 + o(1))$.

\end{lemma}

\begin{proof}
Fix $A,B \subseteq V$ with $|A|=|B|=\alpha n$ and consider the random binomial graph $G=\G(|A|,|B|,p)$. First we consider an auxiliary bipartite graph $H$ on $U\cup W$ such that $U=A\cup A'$, $W=B\cup B'$, $H[A\cup B] = G$, and $H[A'\cup W]$ and $H[U\cup B']$ are complete bipartite graphs. Furthermore, let $s=\log n=|A'|=|B'|$.
We show that $H$ has a perfect matching. It suffices to show that the Hall condition holds, i.e., 
\begin{equation}\label{eq:hallS}
\text{if } S\subseteq U \text{ and } |S| \le |U|/2, \text{ then } |N(S)| \ge |S|,
\end{equation}
and
\begin{equation}\label{eq:hallT}
\text{if } T\subseteq W \text{ and } |T| \le |W|/2, \text{ then } |N(T)| \ge |T|.
\end{equation}
If $|S| < s$, then since $N(S)\supseteq B'$, $|N(S)| \ge |B'|=s \ge |S|$. Therefore, we assume that $s\le |S| \le |U|/2$. Furthermore, we may assume that   $S\cap A' = \emptyset$. For otherwise, $N(S) = W$.
We will show that already for $|S|=s$, $|N(S)| \ge |W|/2=|U|/2$.

Suppose not, that is, $|N(S)| < (\alpha n + \log n)/2$. That means $|B\cap N(S)| < (\alpha n -\log n)/2$, $e(S,B\setminus N(S))=0$ and $|B\setminus N(S)|\ge (\alpha n+\log n)/2>\alpha n/2$.
Observe that the probability that there are sets $S\in A$ and $T\in B$ such that $|S|=s$ and $|T| = \alpha n/2$ and $e(S,T)=0$ is at most
\[
\binom{\alpha n}{s} \binom{\alpha n}{\alpha n/2} (1-p)^{ s  \alpha n/2}
\le 2^{2\alpha n} (1-p)^{ s  \alpha n/2}.
\]
Thus, with probability at most $2^{2\alpha n} \cdot (1-p)^{ s  \alpha n/2}$ the graph $H$ violates \eqref{eq:hallS}, and similarly~\eqref{eq:hallT}.
In other words, with probability at least $1-2\cdot 2^{2\alpha n}  (1-p)^{ s  \alpha n/2}$ the graph $H$ has a perfect matching, and consequently, there is a matching of size $\alpha n - s$ between $A$ and $B$.

Finally, by taking the union bound over all  $A\in \binom{V}{\alpha n}$ and $B\in \binom{V}{\alpha n}$ we get that the probability that there exist $A$ and $B$ such that between $A$ and $B$ there is no matching of size $\alpha n - s$ is at most
\[
\binom{n}{\alpha n} \binom{n}{\alpha n} 2^{2\alpha n+1}  (1-p)^{ s  \alpha n/2} \le 2^n \cdot 2^n  \cdot 2^{2\alpha n+1}  (1-p)^{ s  \alpha n/2} = o(1),
\]
since $s=\log n$. Also, clearly, we get that $\alpha n - s = \alpha n - o(n)$. Thus, \whp~for each $A$ and $B$ there is a matching between $A$ and $B$ of size $\alpha n - o(n)$.
\end{proof}

The next two lemmas deal with matchings with sparse random graphs. For a bipartite graph $G=(A\cup B, E)$ we say that $G$ contains a \emph{$d$-matching} from $A$ to $B$ of size $t$ if there exists in $G$  a subgraph of $t$ vertex-disjoint stars $K_{1,d}$ such that each star is centered in~$A$.

\begin{lemma}\label{lem:nbhs}
For a fixed integer $k$, let $\log n \ll np \le n^{1/(k-1)}$. Suppose the set $A$ satisfies $1 \le |A| \le  (np)^{k-2}$ and the set $B$ has order $\sim n$ and is disjoint from $A$. Consider the random bipartite graph on $A\cup B$ with edge probability $p$.  Then with probability at least $1-\frac{1}{n^3}$ there is a $d$-matching that saturates $A$ with $d = np/4$. 
\end{lemma}

\begin{proof}
To prove this lemma, we verify that Hall's condition holds with high probability. Hence, we show that \whp~every set $S \subseteq A$ has $|N(S)| \ge d|S|$. For convenience, we will denote the orders of $S,~ A,$ and $B$ by their corresponding lowercase letters.  

To that end, we will apply~\eqref{Chernoff_lower} to the random variable $|N(S)| \in \bin(b, q)$ with $q = 1-(1-p)^s$.  But $sp \le ap \le (np)^{k-2}p =(np)^{k-1}/n \le 1$.
Therefore, by~\eqref{eq:ineqs} $q$ may be estimated as $q = 1 -(1-p)^s \ge  1- e^{-sp} \ge 1-(1-sp/2) = sp/2$, and so the expected value satisfies $\mu = bq\ge bsp/2 \ge (1+o(1))snp/2 \gg s\log n$. Now invoking Chernoff's bound~~\eqref{Chernoff_lower}  with $\delta = 1/3$ implies together with the union bound that the probability of failure is at most 
\[
\sum_{S\subseteq A} \Pr(|N(S)| \le (1-\delta)\mu) \le \sum_{s = 1}^a \binom{a}{s} \exp\left(-\delta^2 \mu/2\right)
\le  \sum_{s = 1}^a \exp\left(s \log a  - \mu/18 \right)
\]
and since trivially $a\le n$ and $\mu \gg s\log n$ we can easily bound the above probability by $1/n^3$. Finally note that
\[
|N(S)| \ge (1-\delta)\mu \ge (1+o(1))2snp/6 >  snp/4 = d|S|,
\]
as required.
\end{proof}

\begin{lemma}\label{lem:nbhs_2}
For a fixed integer $k \ge 2$, let $ \sqrt{n} \le np \le \sqrt{n\log n}$. Suppose the set $A$ satisfies $|A| \le  np$ and the set $B$ has order at least $n/2$ and is disjoint from $A$. Consider the random bipartite graph on $A\cup B$ with edge probability $p$.  Then with probability at least $1-\frac{1}{n^3}$ there is a $d$-matching that saturates $A$ with $d = 1/(6p)$. 
\end{lemma}

\begin{proof}
We show that \whp~every set $S \subseteq A$ has $|N(S)| \ge d|S|$. We apply~\eqref{Chernoff_lower} to $|N(S)| \in \bin(b, q)$ with $q = 1-(1-p)^s$. We consider two cases. First assume that $ps \le 1$. Then by~\eqref{eq:ineqs} $q = 1 -(1-p)^s \ge  1- e^{-sp} \ge 1-(1-sp/2) = sp/2$ and so the expected value satisfies $\mu = bq\ge bsp/2 \ge snp/4 \gg s\log n$. So invoking Chernoff's bound~~\eqref{Chernoff_lower}  with $\delta = 1/3$ implies together with the union bound (over all subsets of size at most $1/p$) that the probability of failure is at most 
\[
\sum_{S\subseteq A} \Pr(|N(S)| \le (1-\delta)\mu) \le \sum_{s = 1}^{1/p} \binom{a}{s} \exp\left(-\delta^2 \mu/2\right)
\le  \sum_{s = 1}^{a} \exp\left(s \log a  - \mu/18 \right)
\]
and since trivially $a\le n$ and $\mu \gg s\log n$ we can easily bound the above probability by $1/n^3$. Finally note that
\[
|N(S)| \ge (1-\delta)\mu \ge  snp/6 = s(np)^2/(6np)  =  d|S| (np)^2/n\ge d|S|,
\]
as required. Now assume that $ps > 1$. But then $q \ge 1 - e^{-sp}\ge 1 - e^{-1} \ge 1/2$. Hence $\mu = bq \ge n/4$. So again using Chernoff's bound~~\eqref{Chernoff_lower} and the union bound again (over all subsets of size at least $1/p$) with $\delta = 1/3$, we have that the failure probability is at most 
\[
\sum_{S\subseteq A} \Pr(|N(S)| \le (1-\delta)\mu) \le \sum_{s = 1/p}^{a} \binom{a}{s} \exp\left(-\delta^2 \mu/2\right)
\le  \sum_{s = 1}^{a} \exp\left(s \log a  - \mu/18 \right)
\]
so since $s \log a \le a\log a \le np \log n\le \sqrt{n} (\log n)^{3/2}$ and $\mu \ge n/4$ then we may again bound the failure probability by $1/n^3$. We also satisfy Hall's Condition since 
\[
|N(S)| \ge (1-\delta)\mu \ge  n/6 =  np/(6p)  \ge d|S|.
\]
\end{proof}

The next lemma modifies the well-known result that asserts that \whp~$\G(n,n,p)$ with $p=\frac{\log n+\omega}{n}$ has a perfect matching.

\begin{lemma}\label{lem:omega}
Consider the random bipartite graph $\G(m, m, q)$ with $q = \frac{\log m}{m}$. Let $C>0$ be an absolute constant. Then with probability at least $1-\frac{1}{m^C}$ the graph has a matching of size $(1-o(1))m$.
\end{lemma}

\begin{proof}
The proof goes along the lines of Theorem~6.1 from~\cite{FK}.

Let $G=\G(m, m, q)$ the graph on the set of vertices $A\cup B$.
Let $r=m / (\log m)$. We show that with the desired probability $|N(S)| \ge |S|-r$ for all $S \subseteq A$. This will prove that there is a matching consisting of at least $m-r=(1-o(1))m$ edges. 

Note that if we have an ``obstruction" $|N(S)| < |S|-r$ for $S \subseteq A$, $|S|>(m+r)/2$, then letting $T= B \setminus N(S)$ we have that $|T| = m - |N(S)|$ so $|T| \le m-(|S|-r) < (m+r)/2$ and we have $|N(T)| \le m-|S| < m-(|N(S)|+r) =  |T|-r$. 

Also, if $S\subseteq A$ is a minimal obstruction then every vertex in $N(S)$ has degree at least 2 into $S$ (otherwise we could remove its neighbor from $S$ to obtain a smaller obstruction). Also we may assume that $|N(S)|$ is exactly $s-r-1$ or else we could find a smaller obstruction. 

Thus we look only for obstructions of the following form: $S \subseteq A$ (or $T \subseteq B$) consisting of $r+2 \le s \le (m+r)/2$ vertices, and having a neighborhood consisting of exactly $s-r-1$ vertices, each of which has two neighbors in $S$ (or $T$) and hence there are at least $2(s-r-1)$ edges between $S$ and $N(S)$.
\begin{align*}
   \sum_{s=r+2}^{(m+r)/2} &\binom{m}{s} \binom{m}{s-r-1} \binom{s(s-r-1)}{2(s-r-1)} q^{2(s-r-1)} (1-q)^{s(m-s+r+1)}\\
   &\le  \sum_{s=r+2}^{(m+r)/2} \rfrac{em}{s}^s \rfrac{em}{s-r-1}^{s-r-1} \rfrac{esq}{2}^{2(s-r-1)}  \exp\{-s(m-s+r+1)q\}\\
   & =  \sum_{s=r+2}^{(m+r)/2} \rbrac{ \frac{em}{s} e^{-mq}}^{r+1} \rbrac{\frac{em}{s} \frac{em}{s-r-1} \frac{e^2 s^2 q^2}{4} e^{-(m-s)q}}^{s-r-1}\\
   & \le  \sum_{s=r+2}^{(m+r)/2} \rbrac{\frac{O(1)}{r} }^{r} \rbrac{O(1) \cdot \frac{(\log m)^2 s}{s-r-1}  e^{-(1-s/m)\log m}}^{s-r-1}.
\end{align*}
Now if $r+2\le s\le 2r$, then $\frac{s}{s-r-1} \le 2r$ and 
\[
\rbrac{O(1) \cdot \frac{(\log m)^2 s}{s-r-1}  e^{-(1-s/m)\log m}}^{s-r-1}
\le \rbrac{ (\log m)^2 2r e^{-\frac{1}{2}\log m}}^{s-r-1} 
\le \rbrac{ (\log m)^2 2r e^{-\frac{1}{2}\log m}}^{r}. 
\]
implying
\[
\rbrac{\frac{O(1)}{r} }^r \rbrac{O(1) \cdot \frac{(\log m)^2 s}{s-r-1}  e^{-(1-s/m)\log m}}^{s-r}
= \rbrac{O(1) \cdot (\log m)^2 e^{-\frac{1}{2}\log m}}^{r} = (o(1))^r.
\]
Otherwise, if $s\ge 2r$, then $\frac{s}{s-r} \le 2$ and so $\frac{(\log m)^2 s}{s-r}  e^{-(1-s/m)\log m}=o(1)$ yielding again an upper bound of $(o(1))^r$.

Finally, since in the sum we have only $O(m)$ terms and $r = m/(\log m)$, we trivially get that the failure probability is at most $1/m^C$ for any positive constant~$C$.
\end{proof}

The last lemma deals with colored degrees and codegrees. Let $G=(V,E)$. Denote by $N_i(v)$ the $i$-colored neighborhood of $v$, i.e., the set of vertices $w$ such that $\{v,w\}\in E$ is colored by $i$. In particular, $N(v)$ is the union of $N_i(v)'s$ over all colors~$i$.
Also let $N_{ij}(u,v)$ be the set of all $w \in V$ such that $\{u,w\}$ and $\{v,w\}$ are edges and are colored $i$ and $j$, respectively.

\begin{lemma}\label{lem:sizes}
Let $0 < p < 1$ be a constant and $G=(V,E) = \G(n,p)$. Let $E$ be colored uniformly at random with the colors red and blue denoted by $R$ and $B$, respectively. Then, \whp~for any two vertices $u,v\in V$ we have
\[
|N_R(u)\setminus N(v)| \sim 
|N_B(u)\setminus N(v)| \sim \frac {np}2\left(1-p\right) \text{ and } |N_{RB}(u,v)| \sim |N_{BR}(u,v)| \sim \frac{np^2}{4}.
\]
\end{lemma}
\begin{proof}
Consider any two vertices $u,v \in V$. Note that $|N_R(u) \setminus N(v)|$  is binomially distributed with $n-2$ trials and probability of success $p(1-p)/2$, so the expected value $\mu$ is $(1+o(1))np(1-p)/2$. Set $\delta = \sqrt{7(\log n)/\mu}$ and apply Chernoff's bounds~\eqref{Chernoff_upper} and~\eqref{Chernoff_lower}. Thus, 
\[
\Pr(| |N_R(u)\setminus N(v)| -\mu | \ge \delta \mu )
\le 2\exp(-\delta^2\mu/3) = 2\exp(-(7/3 +o(1))\log n).
\]
Finally, the union bound over all $\binom{n}{2}$ pairs of vertices $u,v \in V$ yields that \whp~for any $u$ and $v$,  $|N_R(u) \setminus N(v)| \sim \frac {np}{2}(1-p)$. 
By symmetry, the same holds for $|N_B(u) \setminus N(v)|$.

For $|N_{RB}(u,v)|$, we apply similar reasoning. Note that $|N_{RB}(u,v)| \in \bin(n-2, p^2/4)$ and so $\mu \sim \frac{np^2}{4}\gg 1$. Applying Chernoff's bounds again with $\delta = \sqrt{7(\log n)/\mu}$ will imply the statement.
\end{proof}

\section{Alternating paths of length two}

Here we prove Theorem~\ref{thm:kappa_r2}. Actually we will prove a more general statement. 
\begin{theorem}\label{thm:kappa_r2_dreg}
Let $G = (V,E)$ be a graph with $|V|=n,$ where all vertices have degree $\sim d$ and every pair of vertices has codegree $\sim d^2/n \gg \log n$. Then, for any number of colors $r\ge 2$
\[
\kappa_{r,2}(G)\sim \left(1-\frac 1r\right) \frac{d^2}n.
\]
\end{theorem}
The standard application of Chernoff's bound implies that $\G(n,p)$ meets the requirements of Theorem \ref{thm:kappa_r2_dreg} \whp~and hence Theorem~\ref{thm:kappa_r2} holds. Furthermore, observe that if $p \ll \sqrt{\log n/n}$, then \whp~the diameter is at least three. Therefore, Theorem~\ref{thm:kappa_r2} is basically optimal.

We separate the proof of Theorem \ref{thm:kappa_r2_dreg} into two lemmas. The upper bound follows from the Cauchy-Schwarz inequality and does not require the assumption on the codegrees. The lower bound is obtained by considering a random coloring along with Chernoff's bound. 

\begin{lemma}\label{lem:kappa_r2_ub}
Let $G = (V,E)$ be a graph where all vertices have degree $\sim d$ and $|V| = n$. Then, for any number of colors $r\ge 2$
\[
\kappa_{r,2}(G) \le \left( 1-\frac{1}{r} +o(1)\right)\frac{d^2}{n}.
\]
\end{lemma}

\begin{proof} 
The total number of alternating paths of length~2 in $G$ is at most (due to the Cauchy-Schwarz inequality)
\begin{align*}
\sum_{v\in V}   \sum_{1\le i<j\le r} \deg_i(v) \deg_j(v) &= \sum_{v\in V} \frac{1}{2} \left(\left( \sum_{i=1}^r \deg_i(v)  \right)^2 -  \sum_{i=1}^r \deg_i(v)^2 \right)\\
&\le \sum_{v\in V}\frac{1}{2} \left(\deg(v)^2 -  \frac{\deg(v)^2}{r}  \right)\\
&\le \frac{nd^2}{2}\left(1 - \frac1r + o(1)\right),
\end{align*}
and so 
\[
\kappa_{r,2}(G) \le \frac{nd^2}{2}\left(1 - \frac1r +o(1)\right)\big{/} \binom{n}{2} =  \left( 1-\frac{1}{r} +o(1)\right)\frac{d^2}{n}.
\]
\end{proof}

\begin{lemma}\label{lem:kappa_r2_lb}
Let $G = (V,E)$ be a graph with $|V|=n,$ where all vertices have degree $\sim d$ and every pair of vertices has codegree $\sim d^2/n \gg \log n$. Then, for any number of colors $r\ge 2$
\[
\kappa_{r,2}(G) \ge \left( 1-\frac{1}{r} +o(1)\right) \frac{d^2}{n}.
\]
\end{lemma}

\begin{proof}
To each edge in $E$ we assign a color from $\{1,\dots,r\}$ uniformly at random.
For $u,v\in V$, let $X_{u,v}$ be the random variable that counts the number of alternating paths between $u$ and $v$ of length~2. Clearly, $X_{u,v} \sim \bin((1+o(1))d^2/n, 1-1/r)$. 

Notice that $\mu:=\E(X_{u,v}) = \left( 1-\frac{1}{r} +o(1)\right)\frac{d^2}{n}\gg \log n$ and $\delta: = \sqrt{(5\log n)/\mu} = o(1)$. Thus, \eqref{Chernoff_lower} yields
\[
\Pr(X_{u,v} \le (1-\delta)\mu) \le \exp(-\mu \delta^2/2) \le \exp((-5/2 + o(1))\log n).
\]
Thus, the union bound taken over all $\binom{n}{2}\le \exp(2\log n)$ pairs of vertices in $V$ yields the statement.
\end{proof}

\section{Alternating paths of length at least three in dense random graphs}

\noindent We consider the case when $p$ is constant and the length  $\ell$ of the paths we seek is greater than two. Clearly the largest possible number of internally disjoint paths between two vertices is $\sim n/(\ell - 1)$. On the other hand, two vertices cannot have more internally disjoint paths between them than one half of the size of the union of their neighborhoods. As it happens, there exists a coloring that allows for a matching lower bound of \emph{alternating} paths, no matter which case occurs. 

Here we prove Theorem~\ref{thm:kappa-dense}, which we state below again for convenience.

\begin{theorem*}{\bf \ref{thm:kappa-dense}}
Let $ 0 < p < 1$ be a constant and $G = \G(n,p)$. Then,  for any integer $\ell \ge 3$ \whp 
\[
\kappa_{r,\ell}(G) \sim \min \left\lbrace\frac{n}{\ell -1},  np\left(1-\frac{p}{2}\right)\right\rbrace.
\]
\end{theorem*}

Our strategy for the lower bound will be to color all edges uniformly at random with two colors. Theorem~\ref{thm:kappa-dense} will follow from Lemmas~\ref{lem:kappa_2ell:c1} and~\ref{lem:kappa_2ell:c2}.

Note that when $\ell = 3$, we have $\min\lbrace n/2, np(1-p/2)\rbrace = np(1-p/2)$ for any $0<p<1$. 

\begin{lemma}\label{lem:kappa_2ell:c1}
Let $0 < p < 1$ be a constant and  $G = \G(n,p)$. Then for any integer $\ell\ge 4$ satisfying $n/(\ell - 1) \le  np(1-p/2)$ for sufficiently large $n$, we have \whp 
\[
\kappa_{r,\ell}(G) \sim \frac{n}{\ell -1}.
\]
\end{lemma}

\begin{proof} 
The upper bound on $\kappa_{r,\ell}(G)$ is obvious. For a matching lower bound, for each pair of vertices $\{x, y\}$ we will have to find a collection of paths using almost all other vertices in the graph. We will accomplish this by first covering  the mutual non-neighbors of $x, y$, and then covering the remaining vertices. We color each edge in $E = E(G)$ uniformly at random either red or blue. 
For any pair of vertices $x,y\in V$, consider the set $U = N(x)\cup N(y)$ and $S = V\setminus U $. Note that \whp~$|U| \sim 2np(1-p/2)$ and $|S| = s \sim n(1-p)^2$. 

Define disjoint sets
\begin{align*}
X_B &= (N_B(x)\setminus N(y))\cup N_{BB}(x,y),\\
X_R &= (N_R(x)\setminus N(y))\cup N_{RR}(x,y),\\
Y_B &= (N_B(y)\setminus N(x))\cup N_{RB}(x,y),\\
Y_R &= (N_R(y)\setminus N(x))\cup N_{BR}(x,y).
\end{align*}
and observe that by Lemma~\ref{lem:sizes}, we have that $|X_B| \sim |X_R|\sim|Y_B|\sim|Y_R|\sim \frac{np}{2}\left(1 - \frac{p}{2}\right)$. 

Now notice that
\begin{equation}\label{eq:dense:1}
np(1-p/2) \ge n/(\ell - 1)
\end{equation}
is equivalent to 
\begin{equation}\label{eq:dense:2}
(\ell - 3)np(1-  p/2 ) \ge n - 2np(1- p/2)
\end{equation}
and hence to 
\begin{equation}\label{eq:dense:3}
np(1- p/2 ) \ge n(1-p)^2/(\ell-3) \sim s/(\ell-3).
\end{equation}
Consequently, $\frac{np}{2}(1- \frac{p}{2} ) \ge  (1+o(1))\frac{s}{2(\ell-3)}$ 
and we can choose
\[
X_B' \subset X_B ,~ X_R' \subset X_R, ~Y_B' \subset Y_B, ~Y_R' \subset Y_R
\] 
such that $|X_B'| \sim |X_R'| \sim |Y_B'| \sim |Y_R'| \sim  \frac{s}{2(\ell - 3)}$.

Now we define a family of disjoint sets $\lbrace W_i\rbrace_{i=1}^{\ell-1}$ and $\lbrace Z_i\rbrace_{i=1}^{\ell-1}$ as follows.  First, we set $W_1 = X_R'$ and $Z_1 = X_B'$. Then for $\ell$ even we define $W_{\ell -1} = Y_B'$ and $Z_{\ell -1} = Y_R'$. For $\ell$ odd we define $W_{\ell -1} = Y_R'$ and  $Z_{\ell -1} = Y_B'$. Note that these sets are disjoint. Now for $2\le i \le \ell - 2$ we inductively and arbitrarily select $W_i,~ Z_i$ of size $s/2(\ell - 3)$ from the remaining vertices of $S$ such that $W_i$ and $Z_i$ are disjoint from the previously selected sets (see Figure~\ref{fig:dense}).

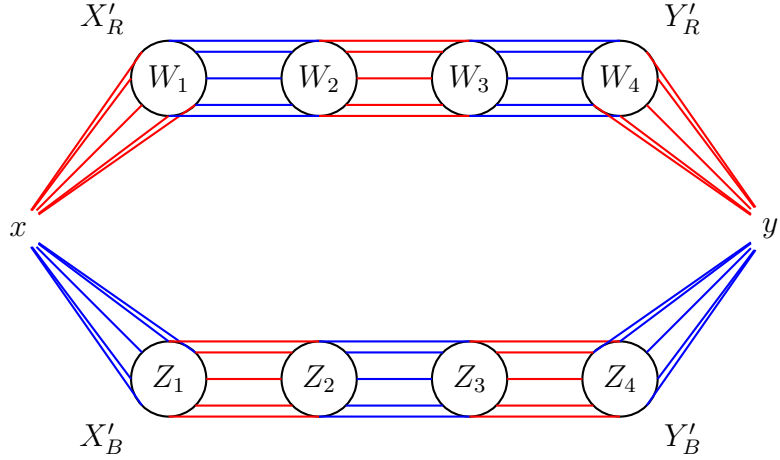
\begin{figure}

\begin{tikzpicture}[thick, black]

\node[black](x) at (0, 0){$x$};
\node[black](y) at (10, 0){$y$};

\begin{scope}
\foreach \i in {1,2,3,4}
	\node[black](w\i) at (\i*2,2) {$W_\i$};
\end{scope}

\node[black,label={[label distance=4mm]135:$X_R'$}](w1) at (2,2) {};
\node[black,label={[label distance=4mm]225:$X_B'$}](z1) at (2,-2) {};
\node[black,label={[label distance=4mm]45:$Y_R'$}](w4) at (8,2) {};
\node[black,label={[label distance=4mm]315:$Y_B'$}](z4) at (8,-2) {};

\begin{scope}
\foreach \i in {1,2,3,4}
	\draw[black] (w\i) circle (.5cm);
\end{scope}

\begin{scope}
\foreach \i in {1,2,3,4}
	\node[black](z\i) at (\i*2,-2) {$Z_\i$};
\end{scope}

\begin{scope}
\foreach \i in {1,2,3,4}
	\draw[black] (z\i) circle (.5cm);
\end{scope}



\begin{scope}
\foreach \theta in {135,180,225,270, 315, 0, 45, 90}
	\coordinate (a\theta) at ({.5*cos(\theta) + 2} ,{.5*sin(\theta) + 2});
\end{scope}

\begin{scope}
\foreach \theta in {135,180,225,270, 315, 0, 45, 90}
	\coordinate (b\theta) at ({.5*cos(\theta) + 4} ,{.5*sin(\theta) + 2});
\end{scope}

\begin{scope}
\foreach \theta in {135,180,225,270, 315, 0, 45, 90}
	\coordinate (c\theta) at ({.5*cos(\theta) + 6} ,{.5*sin(\theta) + 2});
\end{scope}

\begin{scope}
\foreach \theta in {135,180,225,270, 315, 0, 45, 90}
	\coordinate (d\theta) at ({.5*cos(\theta) + 8} ,{.5*sin(\theta) + 2});
\end{scope}


\begin{scope}
\foreach \theta in {135,180,225,270, 315, 0, 45, 90}
	\coordinate (e\theta) at ({.5*cos(\theta) + 2} ,{.5*sin(\theta) - 2});
\end{scope}

\begin{scope}
\foreach \theta in {135,180,225,270, 315, 0, 45, 90}
	\coordinate (f\theta) at ({.5*cos(\theta) + 4} ,{.5*sin(\theta) - 2});
\end{scope}

\begin{scope}
\foreach \theta in {135,180,225,270, 315, 0, 45, 90}
	\coordinate (g\theta) at ({.5*cos(\theta) + 6} ,{.5*sin(\theta) - 2});
\end{scope}

\begin{scope}
\foreach \theta in {135,180,225,270, 315, 0, 45, 90}
	\coordinate (h\theta) at ({.5*cos(\theta) + 8} ,{.5*sin(\theta) - 2});
\end{scope}



\begin{scope}
\foreach \theta in {135,180,225,270, 315}
	\draw[red] (a\theta) -- (x);
\end{scope}

\begin{scope}
\foreach \alpha/\beta in {270/270, 315/225, 0/180, 45/135, 90/90}
	\draw[blue] (a\alpha) -- (b\beta);
\end{scope}

\begin{scope}
\foreach \alpha/\beta in {270/270, 315/225, 0/180, 45/135, 90/90}
	\draw[red] (b\alpha) -- (c\beta);
\end{scope}

\begin{scope}
\foreach \alpha/\beta in {270/270, 315/225, 0/180, 45/135, 90/90}
	\draw[blue] (c\alpha) -- (d\beta);
\end{scope}

\begin{scope}
\foreach \theta in {225,45,0,315, 270}
	\draw[red] (d\theta) -- (y);
\end{scope}


\begin{scope}
\foreach \theta in {45, 90, 135, 180, 225}
	\draw[blue] (e\theta) -- (x);
\end{scope}

\begin{scope}
\foreach \alpha/\beta in {270/270, 315/225, 0/180, 45/135, 90/90}
	\draw[red] (e\alpha) -- (f\beta);
\end{scope}

\begin{scope}
\foreach \alpha/\beta in {270/270, 315/225, 0/180, 45/135, 90/90}
	\draw[blue] (f\alpha) -- (g\beta);
\end{scope}

\begin{scope}
\foreach \alpha/\beta in {270/270, 315/225, 0/180, 45/135, 90/90}
	\draw[red] (g\alpha) -- (h\beta);
\end{scope}

\begin{scope}
\foreach \theta in {135, 90, 45, 0, 315}
	\draw[blue] (h\theta) -- (y);
\end{scope}

\end{tikzpicture}

\caption{Finding the sets $W_i$ and $Z_i$ and corresponding matchings for $\ell=5$.}
\label{fig:dense}

\end{figure}

Now by Lemma \ref{lem:matchings} for each $1\le i \le \ell - 2$ we have \whp~an almost perfect matching between $W_i$ and $W_{i+1}$ as well as $Z_i$ and $Z_{i+1}$ of the appropriate color. That means for odd (even) $i$ we can \whp~find a blue (red) matching of size $s/(2(\ell -3)) + o(n)$ between $W_i$ and $W_{i+1}$ as well as a red (blue) matching of size $s/(2(\ell -3)) + o(n)$ between $Z_i$ and $Z_{i+1}$. This gives us a total of $s/(\ell-3) + o(n)$ alternating paths between $x$ and $y$, covering almost all vertices of $S$ (and also covering some neighbors of $x, y$).

Now we apply this same argument to the set $U' = (X_R\setminus X_R')\cup (X_B\setminus X_B')\cup (Y_R\setminus Y_R')\cup  (Y_B\setminus Y_B')$. Let $k:=|U'| =  2np(1-p/2) - 2s/{(\ell - 3)} + o(n)$. If $k=o(n)$, then $2np(1-p/2) \sim 2s/{(\ell - 3)}$ and by the previous calculations~\eqref{eq:dense:1}-\eqref{eq:dense:3} we obtain that $s/(\ell - 3) \sim n/(\ell - 1)$ yielding $(1+o(1))n/(\ell - 1)$ alternating paths between $x$ and $y$. 

Assume that $k=\Omega(n)$ and note that $|X_B\setminus X_B'| \sim |X_R\setminus X_R'|\sim|Y_B\setminus Y_B'|\sim|Y_R\setminus Y_R'|\sim k/4$. Hence, we can find sets $$X_B'' \subset X_B\setminus X_B' ,~ X_R'' \subset X_R\setminus X_R', ~Y_B'' \subset Y_B\setminus Y_B', ~Y_R'' \subset Y_R\setminus Y_R' $$ such that $|X_B''| = |X_R''| = |Y_B''| = |Y_R''| =  \frac{k}{2(\ell - 1)}$.

As before, we define a family of disjoint sets $\lbrace W_i'\rbrace_{i=1}^{\ell-1}$ and $\lbrace Z_i'\rbrace_{i=1}^{\ell-1}$. We set $W_1' = X_R''$ and $Z_1' = X_B''$. Then for $\ell$ even we define $W_{\ell -1}' = Y_B''$ and $Z_{\ell -1}' = Y_R''$ and for $\ell$ odd $W_{\ell -1}' = Y_R''$ and  $Z_{\ell -1}' = Y_B''$. Now for $2\le i \le \ell - 2$ we inductively and arbitrarily select $W_i',~ Z_i'$ of size $m/2=\frac{k}{2(\ell - 1)}$ from the remaining vertices of $U'$ such that $W_i'$ and $Z_i'$ are disjoint from the previously selected sets. Let us observe that we were able to find all disjoint sets $W_i$, $Z_i$, $W_i'$ and $Z_i'$ since
\begin{align}
2(\ell-1) &\frac{s}{2(\ell-3)} + 2(\ell-1) \frac{k}{2(\ell-1)}
= s\left(1+\frac{2}{\ell-3}\right)+k\nonumber\\
&= s+2np(1-p/2) = n(1-p)^2 +2np(1-p/2) +o(n)= n+o(n). \label{eqn:line12}
\end{align}

Again invoking Lemma \ref{lem:matchings}, for odd (even) $i$ we can \whp~find a blue (red) matching of size $\frac{k}{2(\ell - 1)} + o(n)$ between $W_i'$ and $W_{i+1}'$ as well as a red (blue) matching of size $\frac{k}{2(\ell - 1)} + o(n)$ between $Z_i'$ and $Z_{i+1}'$. So altogether we have found $\frac{k}{\ell - 1} + o(n)$ additional alternating paths between $x$ and $y$. Taken together with the other paths we have found
\begin{align*}
\frac{k}{\ell - 1} + \frac{s}{\ell - 3} + o(n) & = \frac{n}{\ell - 1} + o(n)
\end{align*}
alternating paths between $x$ and $y$. Note that the above equality holds by dividing \eqref{eqn:line12} by $\l-1$.
\end{proof}

\begin{lemma}\label{lem:kappa_2ell:c2}
Let $0 < p < 1$ be a constant and $G = \G(n,p)$. Then for any integer $\ell\ge 3$ satisfying $n/(\ell - 1) \ge  np(1-p/2)$ for sufficiently large $n$, \whp 
\[
\kappa_{r,\ell}(G) \sim np\left(1-\frac{p}{2}\right).
\]
\end{lemma}

\begin{proof} For the upper bound, take any pair of vertices $x,y\in V$ and consider the set $U = N(x)\cup N(y)$. Observe that \whp~$|U|  =  2np(1-p/2) + o(n)$. Furthermore, since the interior of every path of length $\ell$ from $x$ to $y$ must use two vertices from $U$ and $\ell-3$ vertices from elsewhere, the largest possible number of disjoint $xy$-paths possible is $np(1-\frac p2) + o(n)$. 

Now we show the lower bound.
Color each edge in $E$ uniformly at random either red or blue and define disjoint sets
\begin{align*}
X_B = (N_B(x)\setminus N(y))\cup N_{BB}(x,y),\\
X_R = (N_R(x)\setminus N(y))\cup N_{RR}(x,y),\\
Y_B = (N_B(y)\setminus N(x))\cup N_{RB}(x,y),\\
Y_R = (N_R(y)\setminus N(x))\cup N_{BR}(x,y).
\end{align*}
Then by Lemma \ref{lem:sizes}, we have that $|X_B| \sim |X_R|\sim|Y_B|\sim|Y_R|\sim \frac{np}{2}\left(1 - \frac{p}{2}\right)$. We now apply a similar argument to the proof of Lemma \ref{lem:kappa_2ell:c1}. Define $\lbrace W_i\rbrace_{i=1}^{\ell-1}$ and $\lbrace Z_i\rbrace_{i=1}^{\ell-1}$ of size $(np/2)(1-p/2) + o(n)$. First, we set $W_1 = X_R$ and $Z_1 = X_B$. Then for $\ell$ even we define $W_{\ell -1} = Y_B$ and $Z_{\ell -1} = Y_R$; for $\ell$ odd we define $W_{\ell -1} = Y_R$ and $Z_{\ell -1} = Y_B$. Note that these sets are disjoint. Now for $2\le i \le \ell - 2$ we inductively and arbitrarily select $W_i,~ Z_i$ of size $\sim (np/2)(1-p/2)$ from the remaining vertices such that $W_i$ and $Z_i$ are disjoint from the previously selected sets. We are able to find such sets $W_i$ and $Z_i$, since
\[
2(\ell-1)  \frac{np}{2}\left(1 - \frac{p}{2}\right)
= (\ell-1)\cdot  np\left(1 - \frac{p}{2}\right)
\le (\ell-1)\cdot \frac{n}{\ell-1} = n.
\]

We apply Lemma \ref{lem:matchings} to find matchings of the appropriate color between the sets we have defined. For each $1\le i \le \ell - 2$ and odd (even) $i$ we can \whp~find a blue (red) matching of size $(np/2)(1-p/2) + o(n)$ between $W_i$ and $W_{i+1}$ as well as a red (blue) matching of size $(np/2)(1-p/2) + o(n)$ between $Z_i$ and $Z_{i+1}$. This gives us $np(1-p/2) + o(n)$ alternating paths between $x$ and $y$. 

\end{proof}


\section{Alternating paths of length at least three in sparse random graphs}

\noindent In this section we investigate the sparser case $p = o(1)$ and prove Theorem~\ref{thm:kappa-sparse} (stated for convenience below).

\begin{theorem*}{\bf \ref{thm:kappa-sparse}}
Suppose $G = \G(n,p)$ with $p=o(1)$ and $r\ge 2$ is an integer.
\begin{enumerate}[(i)]
\item Let $k\ge 2$ be a positive integer such that $n^{1/k} \le np \le n^{1/(k-1)}$. If $\ell \ge k+2$, then 
\whp~we have $\kappa_{r,\ell}(G)\sim np$.
\item Let $k\ge 2$ be a positive integer such that $(n\log n)^{1/k} \ll np \le n^{1/(k-1)}$. If $\ell = k+1$, then 
\whp~we have $\kappa_{r,\ell}(G)\sim np$.
\item Let $k\ge 3$ be a positive integer such that $(n\log n)^{1/k} \ll np \ll n^{1/(k-1)}$. If $\ell = k$, then 
\whp~we have $\kappa_{r,\ell}(G)= \Theta(n^{k-1}p^k)$.
\end{enumerate}
\end{theorem*}
\noindent
Observe that for $k=2$ in parts \eqref{thm:sparse:ii} and \eqref{thm:sparse:i} condition $p=o(1)$ implies that $np \ll n$. Also notice that in part \eqref{thm:sparse:iii} we may assume that $k\ge 3$; the case $k=2$ follows from Theorem~\ref{thm:kappa_r2}.

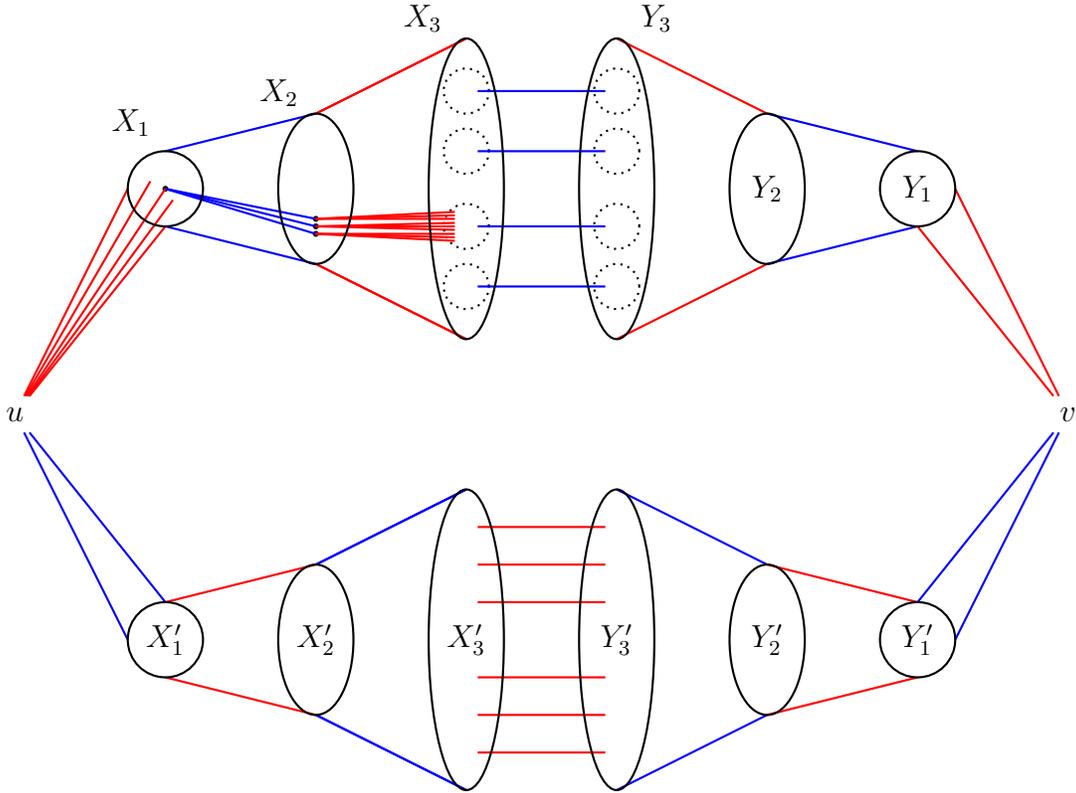
\begin{figure}[h]
\centering
\begin{tikzpicture}[thick, black]

\node[black](x) at (0, 0){$u$};
\node[black](y) at (14, 0){$v$};

\node[black, label={[label distance=.4cm]95:$X_1$}](X1) at (2, 3){};
\node[black, label={[label distance=.8cm]95:$X_2$}](X2) at (4, 3){};
\node[black, label={[label distance=1.8cm]95:$X_3$}](X3) at (6, 3){};

\node[black](Y1) at (12, 3){$Y_1$};
\node[black](Y2) at (10, 3){$Y_2$};
\node[black, label={[label distance=1.8cm]85:$Y_3$}](Y3) at (8, 3){};

\node[black](X1') at (2, -3){$X_1'$};
\node[black](X2') at (4, -3){$X_2'$};
\node[black](X3') at (6, -3){$X_3'$};

\node[black](Y1') at (12, -3){$Y_1'$};
\node[black](Y2') at (10, -3){$Y_2'$};
\node[black](Y3') at (8, -3){$Y_3'$};

\node[black, coordinate](X1t) at (2, 3.5){};
\node[black, coordinate](X1b) at (2, 2.5){};
\node[black, coordinate](X2t) at (4, 4){};
\node[black, coordinate](X2b) at (4, 2){};
\node[black, coordinate](X3t) at (6, 5){};
\node[black, coordinate](X3b) at (6, 1){};
\node[black, coordinate](Y3t) at (8, 5){};
\node[black, coordinate](Y3b) at (8, 1){};
\node[black, coordinate](Y2t) at (10, 4){};
\node[black, coordinate](Y2b) at (10, 2){};
\node[black, coordinate](Y1t) at (12, 3.5){};
\node[black, coordinate](Y1b) at (12, 2.5){};

\draw[blue] (X1t) -- (X2t);
\draw[blue] (X1b) -- (X2b);

\draw[red] (X2t) -- (X3t);
\draw[red] (X2b) -- (X3b);

\draw[red] (X2t) -- (X3t);
\draw[red] (X2b) -- (X3b);

\draw[blue] (Y1t) -- (Y2t);
\draw[blue] (Y1b) -- (Y2b);

\draw[red] (Y2t) -- (Y3t);
\draw[red] (Y2b) -- (Y3b);

\node[black, coordinate](X1t') at (2, -3.5){};
\node[black, coordinate](X1b') at (2, -2.5){};
\node[black, coordinate](X2t') at (4, -4){};
\node[black, coordinate](X2b') at (4, -2){};
\node[black, coordinate](X3t') at (6, -5){};
\node[black, coordinate](X3b') at (6, -1){};
\node[black, coordinate](Y3t') at (8, -5){};
\node[black, coordinate](Y3b') at (8, -1){};
\node[black, coordinate](Y2t') at (10, -4){};
\node[black, coordinate](Y2b') at (10, -2){};
\node[black, coordinate](Y1t') at (12, -3.5){};
\node[black, coordinate](Y1b') at (12, -2.5){};

\draw[red] (X1t') -- (X2t');
\draw[red] (X1b') -- (X2b');

\draw[blue] (X2t') -- (X3t');
\draw[blue] (X2b') -- (X3b');

\draw[blue] (X2t') -- (X3t');
\draw[blue] (X2b') -- (X3b');

\draw[red] (Y1t') -- (Y2t');
\draw[red] (Y1b') -- (Y2b');

\draw[blue] (Y2t') -- (Y3t');
\draw[blue] (Y2b') -- (Y3b');

\node[black, coordinate](X1l) at (1.5, 3){};
\node[black, coordinate](Y1r) at (12.5, 3){};
\node[black, coordinate](X1l') at (1.5, -3){};
\node[black, coordinate](Y1r') at (12.5, -3){};

\draw[red] (x) -- (X1l);
\draw[red] (x) -- (X1b);
\draw[red] (y) -- (Y1r);
\draw[red] (y) -- (Y1b);

\draw[blue] (x) -- (X1l');
\draw[blue] (x) -- (X1b');
\draw[blue] (y) -- (Y1r');
\draw[blue] (y) -- (Y1b');


\node[black](L1) at (6, 4.3){};
\node[black](L2) at (6, 4){};
\node[black](L3) at (6, 3.5){};

\node[black](L4) at (6, 2.5){};
\node[black](L5) at (6, 2){};
\node[black](L6) at (6, 1.7){};

\node[black](R1) at (8, 4.3){};
\node[black](R2) at (8, 4){};
\node[black](R3) at (8, 3.5){};

\node[black](R4) at (8, 2.5){};
\node[black](R5) at (8, 2){};
\node[black](R6) at (8, 1.7){};

\draw[black, dotted] (L1) circle [radius = 3mm];
\draw[black, dotted] (L3) circle [radius = 3mm];

\draw[black, dotted] (L4) circle [radius = 3mm];
\draw[black, dotted] (L6) circle [radius = 3mm];

\draw[black, dotted] (R1) circle [radius = 3mm];
\draw[black, dotted] (R3) circle [radius = 3mm];

\draw[black, dotted] (R4) circle [radius = 3mm];
\draw[black, dotted] (R6) circle [radius = 3mm];


\node[black, coordinate](S) at (4, 2.5){};
\node[black](S1) at (6, 2.45){};
\node[black](S2) at (6, 2.5){};
\node[black](S3) at (6, 2.55){};

\node[black](S4) at (6, 2.4){};
\node[black](S5) at (6, 2.35){};
\node[black](S6) at (6, 2.3){};

\node[black](S7) at (6, 2.6){};
\node[black](S8) at (6, 2.65){};
\node[black](S9) at (6, 2.7){};

\node[black, coordinate](T) at (2, 3){};
\node[black, coordinate](T1) at (4, 2.4){};
\node[black, coordinate](T2) at (4, 2.5){};
\node[black, coordinate](T3) at (4, 2.6){};

\draw[black, fill=black] (T) circle [radius = .25mm];
\draw[black, fill=black] (S) circle [radius = .25mm];
\draw[black, fill=black] (T1) circle [radius = .25mm];
\draw[black, fill=black] (T3) circle [radius = .25mm];

\draw[blue] (T) -- (T1);
\draw[blue] (T) -- (T2);
\draw[blue] (T) -- (T3);

\node[black, coordinate](Tl) at (1.8, 3.1){};
\node[black, coordinate](Tb) at (2.1, 2.85){};

\draw[red] (x) -- (T);
\draw[red] (x) -- (Tl);
\draw[red] (x) -- (Tb);

\draw[blue] (L1) -- (R1);
\draw[blue] (L3) -- (R3);
\draw[blue] (L4) -- (R4);
\draw[blue] (L6) -- (R6);

\draw[red] (S) -- (S1);
\draw[red] (S) -- (S2);
\draw[red] (S) -- (S3);

\draw[red] (T1) -- (S4);
\draw[red] (T1) -- (S5);
\draw[red] (T1) -- (S6);

\draw[red] (T3) -- (S7);
\draw[red] (T3) -- (S8);
\draw[red] (T3) -- (S9);

\node[black](L1') at (6, -4.5){};
\node[black](L2') at (6, -4){};
\node[black](L3') at (6, -3.5){};

\node[black](L4') at (6, -2.5){};
\node[black](L5') at (6, -2){};
\node[black](L6') at (6, -1.5){};

\node[black](R1') at (8, -4.5){};
\node[black](R2') at (8, -4){};
\node[black](R3') at (8, -3.5){};

\node[black](R4') at (8, -2.5){};
\node[black](R5') at (8, -2){};
\node[black](R6') at (8, -1.5){};

\draw[red] (L1') -- (R1');
\draw[red] (L2') -- (R2');
\draw[red] (L3') -- (R3');
\draw[red] (L4') -- (R4');
\draw[red] (L5') -- (R5');
\draw[red] (L6') -- (R6');

\draw (X1) circle [x radius=.5cm, y radius=.5cm];
\draw (X2) circle [x radius=.5cm, y radius=1cm];
\draw (X3) circle [x radius=.5cm, y radius=2cm];

\draw (X1') circle [x radius=.5cm, y radius=.5cm];
\draw (X2') circle [x radius=.5cm, y radius=1cm];
\draw (X3') circle [x radius=.5cm, y radius=2cm];

\draw (Y1) circle [x radius=.5cm, y radius=.5cm];
\draw (Y2) circle [x radius=.5cm, y radius=1cm];
\draw (Y3) circle [x radius=.5cm, y radius=2cm];

\draw (Y1') circle [x radius=.5cm, y radius=.5cm];
\draw (Y2') circle [x radius=.5cm, y radius=1cm];
\draw (Y3') circle [x radius=.5cm, y radius=2cm];

\end{tikzpicture}

\caption{The construction for $\ell = 7$}
\label{fig:sparse}

\end{figure}

\begin{proof}
Due to the low probabilities involved, our strategy is more delicate than in the dense case. Given two fixed vertices, our program is to progressively reveal their colored neighborhoods in an alternating fashion. Due to the sparseness of the graph, these sets are small and can be taken to be disjoint with only minor adjustments. Then we use auxiliary bipartite graphs to find a matching of the appropriate size between the last neighborhoods. 

First we prove \textbf{part~(\ref{thm:sparse:i})}. Let $k\ge 2$ be a positive integer such that $(n\log n)^{1/k} \ll np \le n^{1/(k-1)}$ and $\ell = k+1$. Clearly, $\kappa_{r,\ell}(G)$ is bounded above by the minimum degree, but a Chernoff argument tells us that all degrees in $G$ are \whp~concentrated around the mean of $np(1+o(1))$. So it suffices to establish the lower bound.

We color each edge of $G$ uniformly and independently from the colors red and blue. Fix two vertices $u$ and $v$. We estimate the probability of finding $np(1+o(1))$ disjoint, red-blue paths between $u$ and $v$ and then take the union bound over all choices for $u, v$. To that end, we construct the following sequences of disjoint subsets of $V$: for $1 \le i \le \lceil k / 2 \rceil$ the sets $X_i$, $X_i'$; and for $1 \le i \le \lfloor k / 2 \rfloor$ the sets $Y_i$, $Y_i'$. We require these sequences to have the following properties:

\begin{itemize}
\item The sets $X_1$ and $X_1'$ correspond to the red and blue neighborhoods of $u$ and $Y_1$ and $Y_1'$ correspond to the red and blue neighborhoods of $v$. Hence, by a simple Chernoff argument, their orders $x_1, x_1', y_1, y_1'$ are $\sim np/2$.
\item For $i\ge 2$ we have $x_i = x_1(np/8)^{i-1}$ and there is an $np/8$-matching from $X_{i-1}$ to $X_{i}$ consisting of blue edges for even $i$ and red for odd $i$. Similarly, we have $x_i' = x_1'(np/8)^{i-1}$ and there is an $np/8$-matching from $X_{i-1}'$ to $X_{i}'$ consisting of red edges for even $i$ and blue for odd $i$.
\item For $i\ge 2$ we have $y_i = y_1(np/8)^{i-1}$ and there is an $np/8$-matching from $Y_{i-1}$ to $Y_{i}$ consisting of blue edges for even $i$ and red for odd $i$. Similarly, we have $y_i' = y_1'(np/8)^{i-1}$ and there is an $np/8$-matching from $Y_{i-1}'$ to $Y_{i}'$ consisting of red edges for even $i$ and blue for odd $i$.
\end{itemize}

To construct the sets $X_i$, $X_i'$, $Y_i$, $Y_i'$ we inductively use Lemma \ref{lem:nbhs} by first finding the $X_i$'s and then finding the $Y_i$'s. First, we initialize $X_1$, $X_1'$, $Y_1$, $Y_1'$ as the appropriately colored neighborhoods of $u$ and $v$. Then for $2 \le i \le \lceil k / 2 \rceil - 1$ (this necessarily means that $k\ge 3$), define   
$$
A = X_{i-1}, \quad B = V\setminus \cup_{j = 1}^{i-1} \left(X_j \cup X_j' \cup Y_1 \cup Y_1'\right).
$$

Notice that $A$ has order at most $(np)^{\lceil k /2 \rceil - 2} \le (np)^{k-2}$ and $B$ has order $\sim n$ as required by Lemma \ref{lem:nbhs}. So viewing $A\cup B$ as a random bipartite graph with colored edge probability $p/2$, applying the lemma gives that there exists an $(np/8)$-matching from $A$ to $B$ that saturates $A$ with a failure probability at most $1/n^3$. Set the image of this matching in $B$ to be $X_i$, and observe that $x_i = x_{i-1}(np/8) = x_1(np/8)^{i-1}$. Repeat this process with $A' \cup B'$ for 
$$
A' = X_{i-1}', \quad B' = V\setminus \cup_{j = 1}^{i-1} \left(X_j \cup X_j' \cup Y_1 \cup Y_1'\right)\cup X_i.
$$

Thus, we obtain the sets $X_i$ for $1 \le i \le \lceil k / 2 \rceil$ with the desired property and with a failure probability of $O(1/n^3)$. In a similar manner, for $2 \le i \le \lfloor k / 2 \rfloor$ we obtain $Y_i$, $Y_i'$ with a failure probability of $O(1/n^3)$.

In order to complete our search for $np(1+o(1))$ disjoint, red-blue paths between $u$ and $v$, we find $np(1+o(1))$ correctly colored edges between the sets $X_{\lceil k / 2 \rceil}, X_{\lceil k / 2 \rceil}'$ and $Y_{\lfloor k / 2 \rfloor}, Y_{\lfloor k / 2 \rfloor}'$. If $k$ is odd, we find red edges between $X_{\lceil k / 2 \rceil}$ and $Y_{\lfloor k / 2 \rfloor}'$ as well as blue edges between $X_{\lceil k / 2 \rceil}'$ and $Y_{\lfloor k / 2 \rfloor}$. If $k$ is even, we find red edges between $X_{\lceil k / 2 \rceil}$ and $Y_{\lfloor k / 2 \rfloor}$ as well as blue edges between $X_{\lceil k / 2 \rceil}'$ and $Y_{\lfloor k / 2 \rfloor}'$. For convenience, we will denote $X$ and $Y$ to be the sets with red edges between and $X'$ and $Y'$ to be the sets with blue edges between.

We now construct a random auxiliary bipartite graph $H(X,Y)$ by partitioning $X$ into $\sim np / 2$ disjoint sets of size $(np/8)^{\lceil k / 2 \rceil - 1}$ and $Y$ into $\sim np /2 $ disjoint sets of size $(np/8)^{\lfloor k / 2 \rfloor - 1}$, each set being identified as a vertex in $H(X,Y)$. This partition is done in such a way that each partition class consists of the leaves of a tree rooted at one of the $\sim np/2$ vertices of $X_1$, and internal vertices of this tree are the neighbors of this root in $X_i$ for $2 \le i \le \lceil k / 2 \rceil - 1$ (see Figure \ref{fig:sparse}).

We say there is an edge in $H(X,Y)$ if there is at least one red edge between the corresponding sets of vertices in $G$. Observe that the edge probability is 
\[
q = 1-(1-p/2)^{(np/8)^{\lceil k / 2 \rceil - 1 + \lfloor k / 2 \rfloor - 1}} = 1-(1-p/2)^{(np/8)^{k-2}} \sim 1 - e^{-n^{k-2}p^{k-1}/(2\cdot8^{k-2})}.
\]
Since $n^{k-2}p^{k-1} = (np)^{k-1}/n \le 1$, we get by~\eqref{eq:ineqs} that 
\[
q \sim 1 - e^{-n^{k-2}p^{k-1}/(2\cdot8^{k-2})} \ge 1-(1-n^{k-2}p^{k-1}/(4\cdot8^{k-2})) = n^{k-2}p^{k-1}/(4\cdot8^{k-2}).
\]
We view $H(X,Y)$ as $\G(m,m, q)$ where $m = np/2$. The expected degree of $H(X,Y)$ is given by 
$$
\Theta(mq) = \Theta(n^{k-1}p^k) =\Theta((np)^k/n) \gg \log n = \Omega(\log m).
$$
Hence, by applying Lemma~\ref{lem:omega} with $m,q$ and $C=2k$ we get that there is an almost perfect matching of size $np/2(1+o(1))$ between the bipartition in $H(X,Y)$, with a failure probability of at most $1/ m^{2k}$. We follow the same process with $H(X',Y')$ to obtain an almost perfect matching of size $np/2(1+o(1))$. So altogether we've found $np(1+o(1))$ disjoint alternating paths between our choice of $u$ and $v$. But now taking the union bound over all choices of $u$ and $v$ gives us a total failure probability of at most 
$$
\binom{n}{2}  \cdot O\left(\frac{1}{n^3} + \frac{2}{(np)^{2k}} \right) \sim O\left(\frac{1}{n} + \frac{n^2}{(np)^{2k}}\right) = o(1),
$$
since $(np)^{2k} \gg (n\log n)^2$, by assumption.

Let us summarize what we have found. For any choice of $u$ and $v$, we build four tree structures: two rooted at $u$ and two rooted at $v$. The first level of these trees consists of $\sim np/2$ edges that are completely red or completely blue. Each vertex of this level then has several neighbors (of the opposite color than the first level) that are disjoint from the rest of the tree structures. This pattern continues for each of these neighbors until we have the desired length. Then, looking at the leaves that can be traced back to a single neighbor of $u$, we find at least one edge between these leaves and the leaves in the corresponding tree that can be traced back to a single neighbor of $v$. This gives us our desired $\sim np$ alternating paths between any $u$ and $v$.

Now we discuss how to prove \textbf{part~(\ref{thm:sparse:ii})}. We first assume $k \ge 5$. Here, we may follow the proof of the previous case by fixing a $u$ and $v$ and finding the sets $X_i, X_i'$ for $1 \le i \le \lceil k /2 \rceil$ and $Y_i, Y_i'$ for $1 \le i \le \lfloor k /2 \rfloor$.  
We can achieve this with a failure probability of $O(1/n^3)$.

We must use Lemma~\ref{lem:nbhs} again to find sets $Y_{\lfloor k / 2 \rfloor+1}$ and $Y_{\lfloor k / 2 \rfloor+1}'$ of order $(np/8) y_{\lfloor k / 2 \rfloor}$. 
Observe that for $k\ge 5$ we have
\[
|Y_{\lfloor k / 2 \rfloor+1}| = |Y_{\lfloor k / 2 \rfloor+1}'| = O((np)^{\lfloor k / 2 \rfloor+1}) = O(n^{ \frac{\lfloor k / 2 \rfloor+1}{k-1}}) \ll n,
\]
so we have enough room.
These sets take the role of $Y$ and $Y'$ in the previous part. We first deal with when $\ell = k + a$ for some integer $a > 2$. Here, we alter our construction: we still apply Lemma \ref{lem:nbhs}, which gives us that each vertex has $np/8$ correctly colored neighbors from $X_{\lceil k / 2 \rceil+1}$ to the remaining vertices. But instead, we let $X_{\lceil k / 2 \rceil+1}$ consist of just one vertex from each of the $np/8$-sized stars. Then $X_{\lceil k / 2 \rceil+1}$ has the same order as $X_{\lceil k / 2 \rceil}$. Our previous calculations allow us to continue in this manner, always being sure to alternate the color and accruing a failure probability of $1/n^3$ for each step, until we obtain $X_{\lceil k / 2 \rceil + a  - 2}$, which has order $x_{\lceil k / 2 \rceil}\ll n$.
Notice that if we find an appropriate matching between this set and $Y_{\lfloor k / 2 \rfloor+1}$ (or $Y_{\lfloor k / 2 \rfloor+1}$ depending on the color), this will give us $\sim np/2$ alternating paths of length $\ell = k+ a$.

So we may assume without loss of generality that $\ell = k + 2$. Then we build an auxiliary bipartite graph $H(X, Y)$ as before, but this time we partition $Y$ into $\sim np/2$ vertices of order $(np/8)^{\lfloor k / 2 \rfloor}$ and $X$ into $\sim np/2$ vertices of order $(np/8)^{\lceil k / 2 \rceil - 1}$. And since $n^{k-1}p^k = (np)^k/n\ge 1$, this allows us to say that the edge probability is 
\[
q = 1-(1-p/2)^{(np/8)^{k-1}} \ge 1-e^{-n^{k-1}p^k/(2\cdot 8^{k-1})} \ge 1 - e^{-1 /(2\cdot 8^{k-2})} = \Omega(1).
\]
Set $m=np/2$ and observe that clearly $\Theta(mq) \gg \log m$.
Hence, we may apply the same argument to find \whp~a matching of size $\sim np/2$ between $X$ and $Y$. Similarly for $H(X', Y')$. This gives us the desired $\sim np$ disjoint alternating paths between $u$ and $v$ of length $\ell = k + 2$. 

We still need to address the case $2\le k \le 4$. In these cases, we need to adjust our previous strategy to ensure that the sets in our construction are not too large. We do this first with $3\le k\le 4$. Here if $(n\log n)^{1/k} \ll np \le n^{1/(k-1)}$, then we proceed like in case (ii) so we have no need to create larger sets. Therefore, we may assume that $n^{1/k} \le np \le (n\log n)^{1/k}$. But now the additional sets we need to make are of order
\[
|Y_{\lfloor k / 2 \rfloor+1}| = |Y_{\lfloor k / 2 \rfloor+1}'| = O((np)^{\lfloor k / 2 \rfloor+1}) = O((n\log n)^{ \frac{\lfloor k / 2 \rfloor+1}{k}}) \ll n.
\]

Now we consider the case  $k=2$. Recall that since $p=o(1)$ we have $\sqrt{n} \le np \ll n$ and as before we may assume that $\sqrt{n} \le np \le \sqrt{n\log n}$.
We initialize $X_1, ~X_1', ~Y_1:=Y,~ Y_1':=Y'$ as before. Then we construct a set $X$ and $X'$ of order at most $n/6$ each. We apply Lemma ~\ref{lem:nbhs_2} with $A = X_1$,~ $B = V\setminus (X_1 \cup X_1'\cup Y_1 \cup Y_1')$, then $|A| \le np$ and $|B| \ge n - 4np(1+o(n)) \ge n / 2$. And since $\sqrt{n} \le np \le \sqrt{n\log n}$, we have that there exists an $1/(3p)$-matching between $A$ and $B$ with failure probability $1/n^3$. We call the image of this matching $X$, which has order at most $n / 3$. 
We do the same process with $X_1'$ to obtain $X'$. Then we apply the same strategy as before by constructing $H(X, Y)$ where we have $\sim np/2$ partition classes of size $1/(3p)$ (where each class consists of $1/(3p)$ vertices that are the image of a single vertex in $X$). We partition $Y$ as before. Then since $1/(3p) \ge 1$, the edge probability is 
\[
q = 1-(1-p/2)^{1/(3p)} \ge 1-e^{-1/6} = \Omega(1).
\]
And if we set $m=np/2$ then $\Theta(mq) \gg \log m$.
Hence, we may apply the same argument (using Lemma \ref{lem:omega} with $C = 5$) as before to find \whp~a matching of size $\sim np/2$ between $X$ and $Y$. Similarly for $H(X', Y')$.
If $\ell > 3$, then we apply the same strategy as above by using Lemma \ref{lem:nbhs_2} instead of \ref{lem:nbhs} to finding sets of the same size, before finally a matching between $X$ and $Y$ of the appropriate color.
We finish the proof by taking the union bound over all choices of $u$ and $v$. But again $\binom{n}{2}O(1/n^3 + 1/(np)^{5}) = O(1/n + 1/n^{1/2})= o(1)$ giving us the desired number of alternating paths in this case. \

We now turn our attention to \textbf{part~(\ref{thm:sparse:iii})}. For the upper bound, note that \whp~the number of paths of length $k$ is $O(n^{k+1}p^k)$. Indeed, 
choose an arbitrary vertex $v_0$ and build a copy of $P_{k+1}$ greedily choosing next vertex $v_1$ from $N(v_0)$, $v_2\in N(v_1)$, etc. Since for every vertex $v$, \whp~$|N(v)|\sim np$, the number of all paths of length~$k$ is $O(n (np)^k) = O(n^{k+1}p^k)$ and so there is a pair of vertices with at most $O(n^{k+1}p^k) / \binom{n}{2} = O(n^{k-1}p^k)$ paths.

Now for the lower bound. Fix two vertices $u$ and $v$. We continue as before and again search by applying Lemma~\ref{lem:nbhs} for $X_i, X_i'$ for $1 \le i \le k - 1$. (The assumptions of Lemma~\ref{lem:nbhs} are still satisfied since the last time we applied this lemma to $|A| = O((np)^{k-2})$ and $|B| \sim n$.)
Again, we can achieve this with a failure probability of $O(1/n^3)$. As before, we construct stars $H(X,\{v\})$ and $H(X', \{v\})$ by partitioning $X$ into $\sim np/2$ subsets of order $\Theta((np)^{k-2} )$ and $X'$ also into $\sim np/2$ subsets of order $\Theta((np)^{k-2} )$. Thus, since $p (np)^{k-2} = (np)^{k-1}/n = o(1)$ by assumption,  the edge probability is given by
\[
q = 1-(1-p/2)^{\Theta((np)^{k-2})} = \Theta( n^{k-2}p^{k-1}).
\]
Now, set $m=\Theta(np)$ and observe that the number of edges in $H(X,\{v\})$ as well as in $H(X',\{v\})$ has binomial distribution $\bin(m,q)$ with the expected value 
\[
\mu = mq = \Theta(np \cdot n^{k-2}p^{k-1}) = \Theta((np)^k/n) \gg \log n,
\]
by assumption. Thus, the number of edges in these stars is at least $\mu/2$ with probability at least $1-e^{-\Omega(\mu)}$. Since $\mu \gg \log n$, this probability suffices to overcome the union bound of $\binom{n}{2}$ choices, completing the proof.
\end{proof}

\section{Remarks and further directions}

Here we provide two preliminary results that suggest further directions for study. One involves the parameter $\lambda_{r,\ell}(G)$, introduced in~\cite{BDL}, in which we relax the requirement from $\kappa_{r,\ell}(G)$ that alternating paths between vertices be disjoint. The second result is a pseudorandom analog of Theorem~\ref{thm:kappa-dense}.

\subsection{Not necessarily disjoint paths}

Removing the restriction that alternating paths in $G$ be internally disjoint gives us the number $\lambda_{r,\ell}(G)$, which is the maximum $t$ such that there is an $r$-coloring of the edges of $G$ such that any pair of vertices is connected by $t$ alternating paths of length~$\ell$.
A few results for this number were obtained in \cite{BDL} where it was shown that
\[
\lambda_{2,3}(K_{m,n}) \sim  mn/4
\quad\text{ and }\quad
\lambda_{2,4}(K_{m,n}) \sim  m^2n/8.
\]
Determining $\lambda_{2,\ell}(K_{m,n})$ for general $\ell$ seems to be not an easy problem.
Here we provide the following result for general $G$ and $\ell=3$.

\begin{proposition}\label{prop:lambda}
Let $G$ be a $d$-regular graph. Then $\lambda_{2,3}(G) \le d^3/(4(n-1))$.
\end{proposition}
\begin{proof}
Let the edges of $G$ be 2-colored. Then we claim that the number of all alternating paths of length 3 is at most $nd^3/8$, which will yield the result since
$$
\binom n2 \lambda_{2,3}(G) \le nd^3/8.
$$

Let $G$ be a $d$-regular graph with $G = ([n], E)$ and 2-colored edges. Then let $E = R \cup B$ where $R$ and $B$ are the preimage of red and blue under $c$, respectively. Then we have a red and a blue degree sequence
\begin{align*}
r_1 \le r_2 \le\ldots \le r_n,\\
b_1 \ge b_2 \ge \ldots \ge b_n,
\end{align*}
where $r_i, b_i = d-r_i$ is the red and blue degree of vertex $i$, respectively (under a possible reordering of the vertices). Thus the total number of alternating paths of length $3$ in $G$ is 
\begin{align*}
\sum_{ij\in B} r_i r_j + \sum_{k\ell\in R} b_k b_\ell.
\end{align*}

We claim that this is at most $$\sum_{i\in [n]} \left(r_i^2 \cdot \frac{b_i}{2}+ b_i^2 \cdot \frac{r_i}{2}\right).$$

To do this, we use the following version of the rearrangement inequality (see, for example, \cite{HPL}): For $x_1 \le x_2 \le x_3 \le \ldots \le x_n$ and $y_1 \le y_2 \le y_3 \le \ldots \le y_n$ and any permutation $\pi$ of $[n]$, we have 
$$
\sum_{i=1}^n x_i y_{\pi(i)} \le \sum_{i=1}^n x_i y_i.
$$

\noindent Define $B(i) := \lbrace j \in [n]\setminus \lbrace i\rbrace : ij \in B \rbrace$.
Note that $|B(i)| = b_i$. 
$$
\sum_{ij \in B} r_i r_j \le \frac12 \sum_{i=1}^n \sum_{j \in B(i)} r_i r_j.
$$
We now think of this last sum as $\sum_{i=1}^{m} x_i y_{\pi(i)}$ where $m = b_1+b_2+\ldots+b_n$ and 
\begin{align*}
x_1 = x_2 = x_3 = \ldots = x_{b_1} &:= r_1 \\
x_{b_1 + 1} = x_{b_1+2} = x_{b_1+3} := \ldots = x_{b_1 + b_2} &:= r_2 \\
&\ldots \\
x_{b_1 + \ldots + b_{n-1} + 1} = \ldots = x_{m} &:= r_n, 
\end{align*}
\begin{align*}
y_1 = y_2 = y_3 = \ldots = y_{b_1} &:= r_1 \\
y_{b_1 + 1} = y_{b_1+2} = y_{b_1+3} := \ldots = y_{b_1 + b_2} &:= r_2 \\
&\ldots \\
y_{b_1 + \ldots + y_{n-1} + 1} = \ldots = y_{m} &:= r_n. 
\end{align*}

And observe that $x_1 \le x_2 \le x_3 \le \ldots \le x_m$ and $y_1 \le y_2 \le y_3 \le \ldots \le y_m$ since $r_1 \le r_2\le \ldots \le r_n$. 

But since for $1\le i \le n$ we have that $r_i$ appears in exactly $b_i$ of the sums $\sum_{j\in B(i)} r_j$, then there exists a permutation of $[m]$ that takes the ordering of the indices we get in $\sum_{i=1}^n \sum_{j \in B(i)} r_i r_j$ to $[m]$. Let $\pi$ be the inverse of this permutation. Then by the rearrangement inequality, 
$$
\frac12 \sum_{i=1}^n \sum_{j \in B(i)} r_i r_j = \frac12 \sum_{i=1}^m  x_i y_{\pi(i)} \le \frac12 \sum_{i=1}^m x_i y_i = \frac12\sum_{i=1}^n r_i^2 b_i
$$

We apply the same argument to the $\sum_{k\ell \in R} b_k b_\ell$ with 
\begin{align*}
x_1 = x_2 = x_3 = \ldots = x_{r_n} &:= b_n \\
x_{r_n + 1} = x_{r_n+2} = x_{r_n+3} := \ldots = x_{r_n + r_{n-1}} &:= b_{n-1} \\
&\ldots \\
x_{r_n + \ldots + r_{2} + 1} = \ldots = x_{m'} &:= b_1 
\end{align*}
and
\begin{align*}
y_1 = y_2 = y_3 = \ldots = y_{r_n} &:= b_n \\
y_{r_n + 1} = y_{r_n+2} = y_{r_n+3} := \ldots = y_{r_n + r_{n-1}} &:= b_{n-1} \\
&\ldots \\
y_{r_n + \ldots + r_{2} + 1} = \ldots = y_{m'} &:= b_1, 
\end{align*}
where $m' = r_n + \ldots + r_1$. Then  $x_1 \le x_2 \le x_3 \le \ldots \le x_m$ and $y_1 \le y_2 \le y_3 \le \ldots \le y_m$ since $b_n \le b_{n-1}\le \ldots \le b_1$. 

So we obtain
\[
\sum_{ij\in B} r_i r_j + \sum_{k\ell\in R} b_k b_\ell \le \sum_{i\in [n]} \left(r_i^2 \cdot \frac{b_i}{2}+ b_i^2 \cdot \frac{r_i}{2}\right)
= \sum_{i\in [n]} \frac{r_ib_i}{2}\left(r_i+b_i\right)
= \frac d2\sum_{i\in [n]} r_ib_i.
\]
But this sum is maximized when $r_i = b_i = d/2$ so the total number of alternating paths of length three at most $\frac{d^3n}{8}$.
\end{proof}

It is not difficult to show that the random two-coloring of the edges of $G=\G(n,p)$ implies that \whp~$\lambda_{2,3}(G) \ge (1+o(1))n^2p^3/8$ for $np \gg (n \log n)^{1/3}$. This together with a slightly modified proof of Proposition~\ref{prop:lambda} (where the $d$-regular requirement is replaced by almost $d$-regular) shows that \whp~$\lambda_{2,3}(G) \sim n^2p^3/8$.
It is plausible to believe that the random two-coloring of $\G(n,p)$ always maximizes the parameter $\lambda_{2,\ell}$ for any $\ell\ge 4$.

\subsection{Alternating paths of length at least three in pseudorandom graphs}

We say $G$ is \emph{$d$-pseudorandom} if all degrees are $\sim d$ and all codegrees are $\sim d^2/n$, for $d \gg n^{1/2}$. 
Recall that we have already determined the alternating connectivity for $d$-pseudorandom graphs when we have paths of length two (see Theorem~\ref{thm:kappa_r2_dreg}). 
More generally, we say that $G$ is a \emph{$(n, d, \lambda)$-pseudorandom} if it has $n$ vertices, all vertices have degree $\sim d$, and all eigenvalues except the largest have absolute value at most $\lambda$. It is well known that any $d$-pseudorandom graph $G$ is an $(n, d, o(d))$-graph.

The interested reader can verify that the following result about $\kappa_{r,\ell}$(G) for $G$ a $d$-pseudorandom.

\begin{theorem}\label{thm:kappa_2l}
Suppose $G$ is a $(n, d, \lambda)$-graph, and $\lambda \ll d^2/n$. Then for all fixed $\l \ge 3$ we have 
\[
\kappa_{r,\ell}(G)\sim \min\left\{ \frac{n}{\ell -1}, d - \frac{d^2}{2n}\right\}.
\]
\end{theorem}

Here we follow the same strategy as the proof of Theorem~\ref{thm:kappa-dense} by first coloring each edge randomly and independently and then finding appropriate disjoint sets and colored matchings between them. 
The main difference is that instead of Lemma~\ref{lem:matchings} one can use its analog, given below.

\begin{lemma}\label{lem:match}
Let $G$ be an $(n, d, \lambda)$-graph with $\lambda = o(d)$ and let $A, B$ be disjoint sets of $m \gg (\lambda / d) n$ vertices each. Then $G$ has a matching from $A$ to $B$ containing $\sim m$ edges. 
\end{lemma}
\begin{proof}
The proof is based on an easy application of the Expander Mixing Lemma~\cite{AC} that asserts that if $G$ is an $(n, d, \lambda)$-graph, then for any $S, T \subseteq V(G)$ we have
\begin{equation}\label{eq:eml}
\left| e(S, T) - \frac{d|S||T|}{n} \right| \le \lambda \sqrt{|S||T|} + o\of{\frac{d|S||T|}{n}}.
\end{equation}
Let $\delta = \frac{2\lambda n}{d m}$. By assumption, $\delta =o(1)$. Notice that 
\[
|B \setminus N(S)|= m - |N(S) \cap B| \ge (1+\delta)m - |S|
\]
and $e(S,B \setminus N(S))=0$.
But by~\eqref{eq:eml}, we must have $e(S, T) >0$ whenever $|S||T| \ge 2 \bfrac{\lambda}{d}^2 n^2$. Letting $T = B \setminus N(S)$, we get 
\[
|S||T| \ge |S|[(1+\delta)m - |S|]\ge \delta^2 m^2 = 4\bfrac{\lambda}{d}^2 n^2
\]
which is a contradiction. 
\end{proof}

Here we have an answer for $d \gg (\lambda n)^{1/2}$.  The main obstacle in obtaining an analogous result for smaller $d$ is that we rely on the Expander Mixing Lemma which is ``too strong" in the sense that it is a statement about all sets of vertices $S, T$, and this comes at a price of being ``too weak" in the error term (the little-o term in \eqref{eq:eml}). This is in contrast to random graphs where we are able to handle the sparser cases because we rely only on a similar statement about relatively few (order $n^2$) pairs of sets $S, T$. For $(n, d, \lambda)$-graphs we do not have any analogous tool, i.e. one that has a smaller error term  and still tells us what we need to know about the specific sets $S, T$ that concern our proof. 

\subsection{Remark on the windows used in the sparse case}

Here we comment on the windows of $p$ used in Theorem \ref{thm:kappa-sparse} for parts (\ref{thm:sparse:i}) and (\ref{thm:sparse:iii}). By stipulating that $(n\log n)^{1/k} \ll np$, we avoid when the diameter is changing. Analyzing $\kappa_{r,l}(G)$ for this range seems to require different strategies than the ones we have used here. Thus, we have small gaps where we do not know what is happening. Further, although we have determined the order of magnitude in  (\ref{thm:sparse:iii}), we still leave open the exact constant for this range. An avenue for further work, then, would be to make this result more precise in both of these areas.

\end{document}